\documentclass[12pt]{amsart}
\usepackage{amssymb,amsmath,amscd}

\usepackage{appendix}

\usepackage{amsthm}
\usepackage[T1]{fontenc}

\usepackage{mathrsfs}
\usepackage{color}
\usepackage{epsfig}

\usepackage[all]{xy}

\usepackage{mathtools}

\usepackage{appendix}

\theoremstyle{plain}

\newtheorem*{thm*}{Theorem}

\newtheorem{thm}{Theorem}[section]
\newtheorem{lem}[thm]{Lemma}
\newtheorem{prop}[thm]{Proposition}
\newtheorem{cor}[thm]{Corollary}

\theoremstyle{definition}
\newtheorem{defn}[thm]{Definition}

\theoremstyle{remark}
\newtheorem{rmk}[thm]{Remark}

\newtheorem{example}[thm]{Example}

\newcommand{\ZZ}{\mathbb{Z}}

\newcommand{\RR}{\mathbb{R}}
\newcommand{\CC}{\mathbb{C}}

\newcommand{\KK}{\mathbb{K}}

\newcommand{\SO}{\mathscr{O}}

\newcommand{\ST}{\mathscr{T}}

\renewcommand{\tilde}[1]{\widetilde{#1}}
\renewcommand{\hat}[1]{\widehat{#1}}

\DeclareMathOperator{\Hom}{Hom}

\DeclareMathOperator{\Aff}{Aff}

\begin{document}

\title{Tropical count of curves on abelian varieties}

\author{Lars Halvard Halle}
\address{University of Copenhagen}
\email{larshhal@math.ku.dk}

\author{Simon C. F. Rose}
\address{University of Copenhagen}
\email{simon@math.ku.dk}

\begin{abstract}
We investigate the problem of counting tropical genus \(g\) curves in \(g\)-dimensional tropical abelian varieties. For \(g = 2, 3\), we prove that the tropical count matches the count provided in \cite{gottsche, BL_ab, LS_curves_genus_g} in the complex setting.
\end{abstract}

\maketitle

\section{Introduction}\label{sec:intro}

One of the successes of tropical geometry has been the wide variety of so-called ``correspondence theorems'' that have been produced. These are typically theorems of the following form.

\begin{thm*}[``Correspondence Theorem'']
The tropical count of curves in setting \(X_T\) matches the classical count in setting \(X\), where \(X_T\) is  an appropriate tropicalization of \(X\).
\end{thm*}

Tropical curves are very well-suited to be studied combinatorially, and so with this type of theorem we are given the ability to study and solve many enumerative problems from a fresh perspective. In particular, we have the results of Boehm, Bringmann, Buchholz, and Markwig \cite{BBBM} which uses tropical geometry to prove Mirror Symmetry for the elliptic curve. There is the classic \cite{M_enumerative_p2} which counts tropical curves in toric surfaces. Furthermore, there has been extensive work on studying Hurwitz theory from the tropical perspective, see e.g. \cite{C_tropical_hurw, BCM_double_hurwitz}.

In fact, these correspondence theorems fit more broadly into the framework of the Gross-Siebert program \cite{G_tropical_mirror} which seeks to understand Mirror Symmetry through relating tropical, logarithmic, and classical geometry.

The goal of this paper is to prove the following correspondence theorem. Note that all of the definitions will follow afterwards.

\begin{thm}[Correspondence theorem for abelian varieties]\label{thm_corr}
Let \((A, L)\) be a \(g\)-dimensional complex abelian variety (for \(g = 2, 3\)), with \(L\) an ample line bundle inducing a polarization of type \((d_1, \ldots, d_g)\). Then the number of genus \(g\) curves in the linear system \(|L|\) matches the tropical count.
\end{thm}

More specifically, we prove the following theorem.

\begin{thm}\label{thm_main_A}
Let \(A\) be a tropical torus (of dimension \(g = 2, 3\)) with a line bundle \(L\) such that \((A, c_L)\) is a polarized tropical variety with polarization of type \((d_1, \ldots, d_g)\). Assume that \(A\) is simple and Torelli (Definition \ref{def_Torelli}). Then the number of genus \(g\) curves in the linear system \(|L|\) is \((d_1 \cdots d_g)^2 \nu^\dag(d_1, \ldots, d_g)\), where \(\nu^\dag(d_1, \ldots, d_g)\) is defined in Section \ref{sec_notation}.
\end{thm}

This is compared with \cite{LS_curves_genus_g} which provides the computation in the complex setting. In the simpler case that \(g = 2\) and the polarization is primitive (i.e. of type \((1, n)\)), then we obtain the following, which can be compared with Theorem 3.2 of \cite{gottsche} to obtain the following theorem.

\begin{thm}\label{thm_main}
Let \(A\) be a tropical torus with a line bundle \(L\) such that \((A, c_L)\) is a polarized (tropical) abelian surface with primitive polarization of degree \(n\). Assume that \(A\) contains no (tropical) elliptic curves. Then the number of genus 2 curves in the linear system \(|L|\) is \(n^2\sigma_1(n)\).
\end{thm}

As is usual, counting tropical curves is a somewhat subtle task. One has to understand the combinatorics of the tropical curves as well as the appropriate multiplicity of each curve, and then combine these together. 

In a certain sense, the tropical count is the same as the count of curves over \(\KK\). One could equally interpret the results of this paper to be a proof that counting curves over different algebraically closed fields of characteristic zero are the same.

\subsection{Acknowledgements}

The authors would like to thank Eric Katz for providing us with some ideas and notes to help get this project started. The second author would like to thank K\o benhavns Universitet for hosting him while this project was completed.

\section{Definitions}

\subsection{Notation}\label{sec_notation}

We need the following definitions (see \cite{DEB_kummer}). For a finite abelian group \(G\), let \(G^* = \Hom(G, \KK^\times)\) be its dual group of characters (where \(\KK\) is an algebraically closed field of characteristic zero; more will be said about this later). Furthermore, if \(G  \cong \Lambda / c(X)\) for lattices \(\Lambda, X\) and some homomorphism \(c : \Lambda \to X\), then we define \(G^\dag = X / c^{\dag}(\Lambda)\), where \(c^\dag\) is the adjugate map.

\begin{defn}
For a finite abelian group \(G\), we define
\[
\nu(G) = \sum_{H \leq G} \#\Hom^{sym}(H, H^*)
\]
where \(\Hom^{sym}\) refers to those homomorphisms \(H \to \Hom(H, \KK^\times)\) that are symmetric when viewed as bilinear functions \(H \times H \to \KK^\times\).

In the case that we have \(G \cong \ZZ/d_1\ZZ \times \cdots \times \ZZ/d_g\ZZ\) with \(d_1 \mid \cdots \mid d_g\), then we define
\[
\nu(d_1, \ldots, d_g) = \nu(G).
\]

Finally, we define \(\nu^{\dag}(d_1, \ldots, d_g) = \nu(G^\dag)\) as above. 
\end{defn}

\begin{rmk}
Note that if we let \(n = d_1 \cdots d_g\), then this can be written as
\[
\nu^\dag(d_1, \ldots, d_g) = \nu\big(\tfrac{n}{d_g}, \ldots, \tfrac{n}{d_1}\big).
\]
\end{rmk}

This function satisfies the following property.

\begin{prop}
Let \(G, G'\) be finite abelian groups such that \(\gcd(|G|, |G'|) = 1\). Then
\[
\nu(G \times G') = \nu(G)\nu(G').
\]
Equivalently, if \((d_1, \ldots, d_g)\) and \((d_1', \ldots, d_g')\) are such that \(d_i \mid d_{i + 1}\), \(d_i' \mid d_{i+1}'\), and \(\gcd(d_g, d_g') = 1\), then
\[
\nu(d_1d_1', \ldots, d_gd_g') = \nu(d_1, \ldots, d_g)\nu(d_1', \ldots, d_g').
\]
\end{prop}

\begin{proof}
This follows due to the fact that if \(H, H'\) are coprime order, then
\[
\Hom\big(H \times H', (H \times H')^*\big) = \Hom(H, H^*) \times \Hom\big(H', (H')^*\big)
\]
(which itself follows due to the fact that the product and coproduct in the category of finite abelian groups coincide), and the fact that subgroups of \(G \times G'\) are of the form \(H \times H'\) for subgroups \(H \leq G\) and \(H' \leq G'\).
\end{proof}

We will also use the notation \(\sigma_k(n) = \sum_{d \mid n} d^k\). Furthermore, unless otherwise stated, the notation \(\Hom(-,-)\) will always refer to the Hom-sets in the category of abelian groups.

\subsection{Tropical Tori/Abelian Varieties}\label{subsec-tropicaltori}

A tropical variety can be defined in many ways; as a variety over the min-plus semi-ring, as a certain degeneration of an algebraic variety, or even a variety which locally has integer-affine structure (and whose transition functions preserve that). In any case, all of this simplifies greatly for the case of tropical tori, which have a remarkably simple definition.

In our definition, we follow \cite{MZ_tropical_jacobians} with notation inspired by \cite{katz_unp}. 

\begin{defn}\label{def_trop_ab_var}
Let \(X\) be a rank \(g\) free abelian group. A \(g\)-dimensional tropical torus is given by the quotient
\[
T = \Hom(X, \RR)/\Lambda
\]
where \(\Lambda \hookrightarrow \Hom(X, \RR)\) is a full rank sublattice. Note that the integral structure is given as \(\Hom(X, \ZZ) \subset \Hom(X, \RR)\).
\end{defn}

\begin{rmk}
One main advantage of this definition is that it is basis invariant and provides a more natural definition of the dual torus.
\end{rmk}

In the tropical setting, the sheaf \(\SO\) is replaced by \(\Aff_\ZZ\), the sheaf of affine-linear functions with integral slope. This fits into an exact sequence of sheaves defined on \(T\) given by
\[
0 \to \RR \to \Aff_\ZZ \to \ST_\ZZ^* \to 0
\]
(where the map \(\RR \to \Aff_\ZZ\) is the inclusion of constant functions, and \(\Aff_\ZZ \to \ST_\ZZ^*\) is given by the slope of the map). In particular we obtain a long-exact sequence that is in part given by
\begin{equation}\label{eq_LES_line_bundle}
\cdots H^0(\ST_\ZZ^*) \to H^1(\RR) \to H^1(\Aff_\ZZ) \xrightarrow{c} H^1(\ST_\ZZ^*) \to H^2(\RR) \to \cdots
\end{equation}
We have as expected a bijection between line bundles and \(H^1(\Aff_\ZZ)\) (see \cite{MZ_tropical_jacobians}).

Consider now \(c_L := c(L)\) (where \(c : H^1(\Aff_\ZZ) \to H^1(\ST_\ZZ^*) \cong \Hom(\Lambda, X)\)). That is, \(c_L\) is a map \(c_L : \Lambda \to X\). This naturally induces a pairing \(\Lambda \otimes \Lambda \to \RR\) given by the diagram
\[
\xymatrix{
\Lambda \ar[r]^{c_L}\ar[d]_j & X \\
\Hom(X, \RR)
}
\]
i.e. we define \(\langle \lambda_1, \lambda_2 \rangle := j(\lambda_1)\big(c_L(\lambda_2)\big)\).

\begin{lem}
The pairing \(\langle\, ,\rangle\) is symmetric.
\end{lem}

\begin{proof}
As in the case of  abelian varieties defined over \(\CC\), (see e.g. \cite[Appendix B]{LB_CAV}), we can view elements of \(H^1(\Aff_\ZZ)\) as elements of \(H^1(\Lambda, H^0(\tilde{A}, \Aff_\ZZ))\) i.e. 1-cocycles on \(\Lambda\) with values in \(H^0(\tilde{A}, \Aff_\ZZ)\), where \(\tilde{A} = \Hom(X, \RR)\) is the universal cover of \(A\). These are functions \(\phi : \Lambda \to H^0(\tilde{A}, \Aff_\ZZ)\) which satisfy
\begin{equation}\label{eq_cocycle}
\phi(\lambda_1 + \lambda_2)(f) = \phi(\lambda_1)(\lambda_2 + f) + \phi(\lambda_2)(f).
\end{equation}
Given that an element \(\phi(\lambda)\) of \(H^0(\tilde{A}, \Aff_\ZZ)\) is a globally-defined affine-linear function on \(\Hom(X, \RR)\), we can write such an element as \(\phi(\lambda)(f) = a_\lambda + f\big(c(\lambda)\big)\) with \(c : \Lambda \to X\) (where the last term is of this form since we are considering affine-linear functions with {\em integer} slope). If we then examine the cocycle condition \eqref{eq_cocycle}, we see that the elements \(a_\lambda\) satisfy
\[
a_{\lambda_1 + \lambda_2} = a_{\lambda_1} + a_{\lambda_2} + j(\lambda_1)\big(c(\lambda_2)\big).
\]
Since the left-hand side of this equation is symmetric, it follows that \(j(\lambda_1)\big(c_L(\lambda_2)\big) = j(\lambda_2)\big(c_L(\lambda_1)\big)\) as desired.
\end{proof}

\begin{defn}
Let \(L\) be a line bundle on a tropical torus \(T = \Hom(X, \RR)/\Lambda\). Note that the quotient \(X/c_L(\Lambda)\) is a finite abelian group, and is hence isomorphic to \(\ZZ/d_1\ZZ \times \cdots \times \ZZ/d_g\ZZ\) for some integers \(d_1 \mid \cdots \mid d_g\). We define the {\em type} of \(L\) to be the tuple \((d_1, \ldots, d_g)\). We define the {\em degree} of \(L\) to be the index \([X : c_L(\Lambda)] = d_1 \cdots d_g\). Finally, we say that the polarization is {\em primitive} if \(\gcd(d_1, \ldots, d_g) = 1\).
\end{defn}

\begin{rmk}
If \(c_L\) is a polarization of type \((d_1, \ldots, d_g)\), then we will write \(n = d_1 \cdots d_g\).
\end{rmk}

We will now define a (polarized) tropical abelian variety.

\begin{defn}\label{definition-tropicalAV}
Let \(A\) be a tropical torus together with a line bundle \(L\) such that \(c_L\) induces a positive definite bilinear form. We call the pair \((A, c_L)\) a polarized tropical abelian variety with polarization \(c_L\). The degree of the polarization is the degree of \(c_L\). If \(c_L\) has degree 1, then we call the polarization principal.
\end{defn}

\begin{rmk}
Note that if \(A = \Hom(X, \RR)/\Lambda\) is principally polarized by \(c_L\), then we have an isomorphism \(c_L : \Lambda \to X\). Consequently, we can write any principally polarized abelian variety as \(\Hom(\Lambda, \RR)/\Lambda\).
\end{rmk}

We will next investigate maps between tropical tori. The key is that such maps must preserve the integral structure of the tori, which greatly restricts their form.

We will start with defining those morphisms that preserve the identity of the tori. 

\begin{defn}
Let \(T_1 = \Hom(X_1, \RR)/\Lambda_1\) and \(T_2 = \Hom(X_2, \RR)/\Lambda_2\) be tropical tori. A homomorphism \( f : T_1 \to T_2\) consists of a pair of morphisms \((g, h)\) where
\[
g : \Lambda_1 \to \Lambda_2 \qquad \qquad h : X_2 \to X_1
\]
such that the following diagram commutes.
\[
\xymatrix{
\Lambda_1 \ar[d] \ar[r]^g & \Lambda_2 \ar[d] \\
\Hom(X_1, \RR) \ar[r]_{h^*} & \Hom(X_2, \RR).
}
\]

Note of course that this is a necessary condition\footnote{This isn't quite correct; any map \(h_\RR : X_2 \otimes \RR \to X_1 \otimes \RR\) would actually suffice, but this is not relevant in our case.} for this to yield a map on the level of topological spaces \(T_1 \to T_2\). It is also sufficient when we consider that the map must preserve the underlying integral structure.

Given such a map, we define the {\em topological degree} \(d_t\) to be the index \([\Lambda_2 : g(\Lambda_1)]\), and we define the {\em metric degree} \(d_m\) to be the index \([X_1 : h(X_2)]\). We define the {\em tropical degree} of the map to be the product \(d_t d_m\).
\end{defn}

In general, a morphism \(T_1 \to T_2\) is the composition of a homomorphism as defined above together with a translation in \(a \in T_2\). These are given by post-composing a homomorphism \(f : T_1 \to T_2\) with the morphism
\[
t_a : T_2 \to T_2 \qquad x \mapsto x + a.
\]

\begin{rmk}
We will use the term {\em homomorphism} to specifically refer to a morphism which preserves the identity, whereas a morphism may include a translation.
\end{rmk}

We next investigate how polarizations are affected by morphisms. So let \(A_1, A_2\) be tropical tori, \(A_i = \Hom(X_i, \RR)/\Lambda_i\). Translations do not affect polarizations, so we may assume that a morphism is in fact a homomorphism; let \(f = (g, h)\) be a such morphism between them. Let \(c_L\) be a polarization on \(A_2\). Then the induced polarization on \(A_1\) is given by
\[
\xymatrix{
\Lambda_1 \ar@/_1em/[rrr]_{f^*c_L} \ar[r]^g & \Lambda_2 \ar[r]^{c_L} &  X_2 \ar[r]^h & X_1
}.
\]

Lastly, let us define the dual torus/abelian variety.

\begin{defn}
Let \(A = \Hom(X, \RR)/\Lambda\) be a tropical torus. We define the dual torus to be \(\hat{A} = \Hom(\Lambda, \RR)/X\), where the inclusion \(X \hookrightarrow \Hom(\Lambda, \RR)\) is given by
\[
\Hom\big(\Lambda, \Hom(X, \RR)\big) \equiv \Hom(\Lambda \otimes X, \RR) \equiv \Hom\big(X, \Hom(\Lambda, \RR)\big).
\]
Moreover, if \(A\) is polarized by a degree \(n\) polarization \(c_L\), then the dual polarization is given by \(c_{\hat{L}} = c_L^\dag\) (the adjugate of \(c_L\)), and satisfies
\[
c_L \circ c_{\hat{L}} = n \cdot id_X \qquad \qquad c_{\hat{L}} \circ c_L = n \cdot id_\Lambda.
\]
Note that the dual polarization is of type \(\big(\tfrac{n}{d_g}, \cdots, \tfrac{n}{d_1}\big)\). See by analogy over \(\CC\) \cite[Proposition 2.7]{LB_dual_polarization}. 
\end{defn}

Note that unlike the definition of dual polarization provided in \cite{LS_curves_genus_g}, we do {\em not} have that \(\hat{\hat{L}} \cong L\). The definition we provide (which is the same as in \cite{LB_dual_polarization}) is more natural in our context, and the count of curves we obtain is entirely equivalent. 

This can be explained as follows. If we let \(L'\) be the dual line bundle as defined in \cite{LS_curves_genus_g}, then we have the relation that \((L')^{\otimes k} = \hat{L}\) for some \(k \geq 1\). However, the count of curves in the linear system defined by \(L\) in the abelian variety \(A\) (all defined over \(\CC\)) is obtained by looking at certain groups defined by the line bundle \((L')^{\otimes k}\), and so this is exactly the same line bundle that we use.

\vspace{0.3 cm}

Next, given a polarized tropical abelian variety \((A, c_L)\), we obtain a natural homomorphism \(A \to \hat{A}\) given by the pair \((c_L, c_L)\). That is, we have
\[
\xymatrix{
\Lambda \ar[r]^{c_L}\ar[d] & X \ar[d] \\
\Hom(X, \RR) \ar[r]_{c_L^*} & \Hom(\Lambda, \RR)
}
\]
which has tropical degree \(n^2\) (each of the topological and metric degrees are \(n\), respectively). Moreover, the compositions \(A \to \hat{A} \to A\) and \(\hat{A} \to A \to \hat{A}\) are multiplication by \(n\).

\begin{rmk}[See \cite{MZ_tropical_jacobians}]\label{rmk_translation}
One can equivalently define \(\hat{A} = Pic^0(A)\), which by the long-exact-sequence \eqref{eq_LES_line_bundle} is isomorphic to our dual. Moreover, the map \(A \to \hat{A}\) is in this case given by \(a \mapsto L^{-1} \otimes t_a^* L\) as in the complex case.

Note that we have a natural isomorphism \(P \to \hat{P}\) if \(P\) is principally polarized, which allows us to identify the two tori.
\end{rmk}

We will provide one last definition which is important for our case.

\begin{defn}\label{def_simple}
Let \(T\) be a tropical torus of dimension \(g\). Then we say that \(T\) is {\em simple} if it does not admit any subtori of dimension \(0 < h < g\).
\end{defn}

If \(A\) is a tropical abelian variety of dimension \(g = 2, 3\), then it is simple if and only if there are no non-constant maps from elliptic curves to \(A\).

\subsection{Tropical Curves}

We will work in a simplified setting which avoids a possible weighting function on vertices (which corresponds to certain degenerations of tropical curves). This section will be particularly brief; for more detail, see \cite{MZ_tropical_jacobians}.

\begin{defn}
Let \(\Gamma\) be a graph such that \(h^1(\Gamma) = g\) together with a function \(\ell : e(\Gamma) \to \RR_{>0}\) (a so-called {\em metric} graph). Furthermore, we assume that \(\Gamma\) has no 2-valent vertices. We call such a pair a genus \(g\) abstract tropical curve. 
\end{defn}

\begin{rmk}
Note that the integral structure comes from viewing an edge as \([0, \ell(e)] \subset \RR\).
\end{rmk}

\begin{defn}
Let \(\Gamma\) be the graph underlying a tropical curve. We say that \(\Gamma\) is {\em \(m\)-edge connected} if the removal of any \(k < m\) points in the interior of distinct edges leaves a connected graph.
\end{defn}

Maps between tropical varieties are in general difficult to define. When the source is a curve and the target an abelian variety, then this is simpler.

\begin{defn}\label{def_tropical_map}
A morphism from a tropical curve \(\Gamma\) to a tropical abelian variety \(A\) is
\begin{itemize}
\item A continuous map \(f : \Gamma \to A\) that is locally affine-linear on the edges
\item A weighting of the edges (i.e. a function \(w : E(\Gamma) \to \ZZ_{\geq 0}\))
\end{itemize}
such that for each vertex of \(\Gamma\) with outgoing tangent vectors \(\{v_i\}\) we have \(f_*v_i = w_i \xi_i\) for some primitive integral vectors \(\{\xi_i\}\), and which satisfy
\[
\sum_i w_i \xi_i = 0
\]
(the {\em balancing condition}).
\end{defn}

\begin{example}
Some common examples of local models are
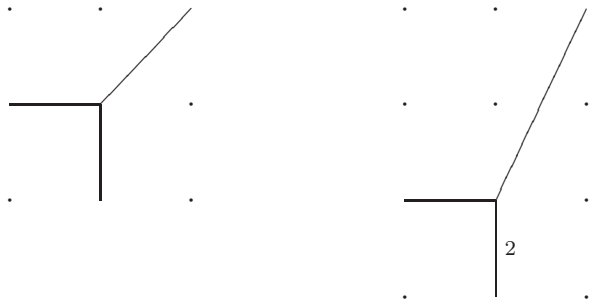
\[
\xymatrix{
\cdot & \cdot & *{} \\
*{} & *{} \ar@{-}[l]\ar@{-}[d]\ar@{-}[ur] & \cdot \\
\cdot & *{} & \cdot
}
\qquad \qquad \qquad
\xymatrix{
\cdot & \cdot & *{} \\
\cdot & \cdot & \cdot \\
*{} & *{} \ar@{-}[l]\ar@{-}[d]^2\ar@{-}[uur] & \cdot\\
\cdot & *{} & \cdot
}
\]
where the weight \(w_e\) of an edge (if greater than 1) is written next to it.
\end{example}

\begin{defn}
The Jacobian of a tropical curve is defined in an analogous way to the Jacobian of a curve defined over \(\CC\). We have a notion of the space of 1-forms on \(\Gamma\) given by \(\Omega(\Gamma)\). There is a map \(H_1(\Gamma,\ZZ) \hookrightarrow \Omega(\Gamma)^\vee\) into the dual of \(\Omega(\Gamma)\) given by integrating over \(1\)-cycles. From this, we define
\[
J(\Gamma) = \Omega(\Gamma)^\vee/H_1(\Gamma, \ZZ). 
\]
\end{defn}

There is of course a map \(\Gamma \to J(\Gamma)\) which is given by integrating along partial paths; that is, for \(p \in \Gamma\), choose any path \(\gamma : p_0 \to p\), which yields a map
\[
\omega \mapsto \int_{p_0}^p \omega
\]
which is well defined up to \(H_1(\Gamma, \ZZ)\), and so we have a map \(\Gamma \to J(\Gamma)\) as claimed.

Most importantly for our purposes, the Jacobian satisfies the following universal property whose proof is nearly verbatim as the one over \(\CC\).

\begin{prop}
Let \(A\) be a tropical abelian variety, let \(\Gamma\) be a tropical curve, let \(p_0 \in \Gamma\), and let \(f : \Gamma \to A\) be a map such that \(f(p_0) = 0 \in A\). Then there is a unique factorization
\[
\xymatrix{
\Gamma\ar[rr]^f\ar[dr] & & A \\
& J(\Gamma) \ar@{-->}[ur]_{\tilde{f}}
}
\]
through $J(\Gamma)$.
\end{prop}

Using the universal property we can now make the following central definition. Note that in \cite{gottsche} (and in \cite{BL_ab} as well) we restrict ourselves to curves in a fixed linear system \(|L|\). In the case of abelian varieties defined over \(\CC\), this is simply the statement that \(\SO\big(f(C)\big) = L\). While we can use the same approach in the tropical case, there is a simpler definition in our case using the universal property which suits us better.

\begin{defn}\label{definition-linsys}
Let \(f : C \to A\) be a morphism from a tropical curve of genus \(g\) to an abelian variety of dimension \(g\), and let \(L\) be a line bundle on \(A\). Let \(J(C)\) be the Jacobian of \(C\) and let \(\Theta_C\) be the unique theta divisor on \(J(C)\). Then we say that the image of \(C\) in \(A\) is in the linear system \(|L|\) if
\[
\tilde{f}^*L \cong n \Theta_C
\]
where \(n\) is the degree of \(L\). Equivalently, we have that \(\hat{f}^*\Theta = \hat{L}\), where \(\hat{f} : \hat{A} \to J(C)\) is the dual morphism, and \(\hat{L}\) is the dual line bundle.
\end{defn}

Finally, we introduce a genericity condition that we will assume holds for \(A\).

\begin{defn}\label{def_Torelli}
Let \(A\) be a tropical abelian variety. We will say that \(A\) is {\em Torelli 3-connected} or more simply {\em Torelli} if it is not isogenous to the Jacobian of a curve which is not 3-connected.
\end{defn}

The purpose of this condition is the following. As per \cite[Corollary 4.1.16]{viv_torelli}, if we stay away from this locus then the Torelli map is injective. In particular, the Jacobian uniquely determines the curve. Moreover, one can verify (by the construction of the 3-edge connectivization of a curve) that the abelian varieties which do not satisfy this condition (i.e. are isogenous to Jacobians of non-3-connected curves) form a positive codimension locus in the moduli space of tropical abelian varieties, and so a suitably generic abelian variety will always be Torelli.

Note further that if \(A\) is a \(g\)-dimensional abelian variety which is simple and Torelli, and if \(f : \Gamma \to A\) is a map from a genus \(g\) tropical curve, then it follows that \(\Gamma\) must be 3-connected (and so by \cite[Theorem 4.1.9]{viv_torelli}, the Jacobian of the curve determines the curve).

\subsection{Rigid Analytic Varieties}
Throughout this subsection, $\mathbb{K}$ denotes a non-archimedean field with absolute value $ \vert \cdot \vert $ and valuation ring $R$. We shall assume that $\mathbb{K}$ is algebraically closed, complete, and of characteristic zero (though this is not essential for many of the statements below).

\subsubsection{}

As we already indicated in Section \ref{sec:intro}, our main result, Theorem \ref{thm_main_A}, is really a statement concerning the count of curves in a linear system $ \vert L \vert $ on an abelian $\mathbb{K}$-variety $A$, where $L$ is an ample line bundle on $A$ of degree $n$ ($C$ being in $\vert L \vert $ is defined exactly as in Definition \ref{definition-linsys}). In order to translate statements between algebraic $\mathbb{K}$-varieties and tropical varieties, we shall pass through the category of \emph{rigid analytic} varieties. In particular, this enables us to formulate the necessary geometric conditions on $A$ for which this count matches the complex count. Moreover, this will allow us to use the powerful technique of rigid analytic uniformization of abelian varieties. 

\vspace{0.3 cm}

For an introduction to rigid analytic geometry, which is particularly well suited for questions concerning curves and abelian varieties, we refer to the excellent textbook \cite{Rigid-book}. In particular, the reader will there find a thorough discussion of the various standard facts we recall in the paragraphs below.

\subsubsection{Analytification}\label{subsec:rigidGAGA}
The rigid analytification functor, denoted $(\cdot)^{\mathrm{an}}$, takes a (proper) $\mathbb{K}$-variety $X$ to its analytification $ X^{\mathrm{an}}$, which is a (proper) rigid-analytic space over $\mathbb{K}$. We shall use the following well-known properties of this functor: 
\begin{itemize}

\item The functor $(\cdot)^{\mathrm{an}}$ is fully faithful on the category of \emph{proper} $\mathbb{K}$-varieties.

\item If $X$ is a proper $\mathbb{K}$-variety, there is also an analytification functor $\mathcal{F} \mapsto \mathcal{F}^{\mathrm{an}}$ on $\mathcal{O}_X$-modules, which yields an equivalence from the category of coherent sheaves on $X$ to the category of coherent sheaves on $X^{\mathrm{an}}$.

\end{itemize}

\subsubsection{Uniformization of abelian varieties}\label{subsec:rigidunif}

Any abelian $\mathbb{K}$-variety $A$ allows a uniformization in the rigid analytic category. Here, by a uniformization of $A$, we mean the following data: 
\begin{itemize}
\item A semi-abelian variety $E$, which is an extension
$$ 0 \to T \to E \to B \to 0 $$
of a torus $T$ by an abelian variety $ B $ having \emph{good reduction} over $R$.

\item A lattice $ M \subset E $ of rank equal to $ \mathrm{dim}(T)$ and an exact sequence
$$ 0 \to M^{\mathrm{an}} \to E^{\mathrm{an}} \to A^{\mathrm{an}} \to 0. $$
\end{itemize}

Be aware that the morphism $E^{\mathrm{an}} \to A^{\mathrm{an}}$ only exists analytically, even though $A$ and $E$ are algebraic. In the particular case where the abelian part $B$ of $E$ is zero, we shall say that $A$ is uniformized by a torus; this turns out to be the relevant case for our purposes.

\subsubsection{Analytification of curves and abelian varieties}
A smooth proper connected rigid $\mathbb{K}$-group carrying an \emph{ample} line bundle is the analytification of an abelian variety. Thus, we shall speak about abelian varieties also in the rigid setting. In this case, an ample line bundle is equivalent to a \emph{polarization}. 

\vspace{0.2 cm}

For later use, we include some details concerning polarizations of abelian varieties uniformized by tori. Let $A$ be an abelian variety uniformized by a torus $ T = \mathrm{Spec}~\mathbb{K}[\hat{M}] $ modulo a lattice $ M \subset T $. 
Then a polarization of $A$ is determined by an injective linear map $ c_L \colon M \to \hat{M} $ for which the induced bilinear form $ \langle m_1, c_L(m_2) \rangle  $ is symmetric and $ \vert \langle m, c_L(m) \rangle \vert < 1 $ for all $m \neq 0 $ (i.e., it is positive definite). The degree of $c_L$ is the number $ [ \hat{M} : c_L(M) ] $. The polarization induces a morphism from $A$ to its dual $ \hat{A} $, which is uniformized by $ \hat{T} = \mathrm{Spec}~\mathbb{K}[M] $ modulo the lattice $ \hat{M} \subset \hat{T}$.

\vspace{0.2 cm} 

We also record the fact that if $C$ is a smooth projective $\mathbb{K}$-curve, then the formation of the Jacobian variety commutes with analytification, i.e., $J(C^{\mathrm{an}}) = J(C)^{\mathrm{an}}$. Moreover, the analytic Jacobian again carries a principal polarization.

\subsubsection{From algebraic to analytic}

Let $A$ be a simple abelian variety over $\mathbb{K}$, equipped with a polarization $L$ of degree $n$. We assume that $A$ is uniformized by  a torus. By the properties of the analytification functor discussed in the previous paragraphs, the count of morphisms $ f \colon C \to A $ in the linear system $ \vert L \vert $ can be performed instead in the analytic category. The advantage is that we can now use uniformization; this in fact reduces our problem to a certain count of lattices.

Our assumptions on $A$ imply, for a morphism $f$ as above, that the induced morphism $ \tilde{f} \colon J(C) \to A $ is an isogeny.  Moreover, the Jacobian $J(C)$ of $C$ is necessarily simple. This has the immediate pleasant consequence, as we explain in Lemma \ref{lemma:torusunif} below, that also $J(C)$ is uniformized by a torus.

\begin{lem}\label{lemma:torusunif}
Let $(P,\Theta)$ be a principally polarized abelian variety and let $ g \colon P \to A $ be an isogeny. Then $P$ is uniformized by a torus as well.
\end{lem}
\begin{proof}

We identify $ P = \hat{P} $, and denote by $ \hat{g}$ be the dual isogeny of $g$.
The composition $ g \circ \hat{g} $ yields an isogeny $ \hat{A} \to A $. Analytically, this isogeny corresponds to an \emph{injective} homomorphism $ \hat{M} \to M $ which has to factor through the lattice $N$ appearing in the uniformization of $P$. This is only possible if $\mathrm{rk}(N) = \mathrm{dim}(P)$, which implies that $P$ is also uniformized by a torus.

\end{proof}

\section{Proof of main theorem}

We are now prepared to prove the main theorem. We will break this up into two steps; first, we will look at the combinatorics of the tropical maps, and we will count these without any reference to multiplicities. Second, we will compute the multiplicity to obtain the result.

\subsection{Na\"ive Count of Curves}

We will prove the following proposition.

\begin{prop}\label{prop_naive_count}
Let \(A =\Hom(X, \RR)/\Lambda\) be a tropical torus of dimension \(g\) with a line bundle \(L\) such that \((A, c_L)\) is a polarized (tropical) abelian variety with  polarization of type \((d_1, \ldots, d_g)\). Assume that \(A\) is simple and Torelli. Then the number of genus \(g\) curves in the linear system \(|L|\), not including multiplicity, is \(n \cdot \#\{H \leq \Lambda / c_L^\dag(X)\}\).
\end{prop}

The proof of Proposition \ref{prop_naive_count} breaks up into three steps.

\begin{prop}\label{prop_12}
Define the two sets
\begin{gather*}
S_1 = \{ f : \Gamma \to A \mid \Gamma \in |L|, g(\Gamma) = g\}\\
S_2 = \{ F : (P, c) \to (A, c_L) \mid c \text{ is principal}, F^*c_L = n \cdot c, F(0_P) = 0_A\}.
\end{gather*}
There is then a bijection
\[
S_1 \iff \ker(A \to \hat{A}) \times S_2.
\]
\end{prop}

\begin{prop}\label{prop_23}
Let \(S_2\) be defined as above, and let \(S_3\) be defined as
\[
S_3 = \{ X \xrightarrow{f_1} I \xrightarrow{f_2} \Lambda \mid rank(I) = g, f_2 \circ f_1 = c_L^\dag\}.
\]
Then there is a bijection \(S_2 \iff S_3\).
\end{prop}

\begin{prop}\label{prop_3sigma}
We have that
\[
|S_3| = \#\big\{H \leq \Lambda/c_L^\dag(X)\}.
\]
\end{prop}

\begin{cor}
We have that
\[
|S_1| = n \cdot \#\big\{H \leq \Lambda/c_L^\dag(X)\big\}.
\]
\end{cor}

\begin{proof}
This follows from the Propositions \ref{prop_12}, \ref{prop_23}, and \ref{prop_3sigma}, combined with the fact that the topological degree of the map \(A \to \hat{A}\) is \(n = d_1\cdots d_g\), and hence its kernel consists of \(n\) elements.
\end{proof}

Let us prove the propositions in reverse order.

\begin{proof}[Proof of Proposition \ref{prop_3sigma}]
The set \(S_3\) consists of factorizations
\[
\xymatrix{
X \ar[r]^{f_1} \ar@/_1em/[rr]_{c_L^\dag} & I \ar[r]^{f_2} & \Lambda
}
\]
If we quotient out each term by \(X\) we end up with
\[
0 \to I/X \to \Lambda/X.
\]
We see that such a factorization is equivalent to a subgroup of \(\Lambda / c_L^\dag(X)\) as claimed.
\end{proof}

\begin{proof}[Proof of Proposition \ref{prop_23}]
Let \(F : P \to A\) be an element in \(S_2\), with \(P = \Hom(I, \RR)/I\). We have a dual morphism \(\hat{F} : \hat{A} \to P\) which yields a diagram
\[
\xymatrix{
X \ar[d]\ar[r]_{f_1}\ar@/^1em/[rr]^{c_L^\dag} & I \ar[d]\ar[r]_{f_2} & \Lambda\ar[d] \\
\Hom(\Lambda, \RR) \ar[r]_{{f_2}^*} & \Hom(I, \RR) \ar[r]_{f_1^*} & \Hom(X, \RR)
}
\]
for some pair of homomorphisms \((f_1, f_2)\). This yields an element of \(S_3\).

Conversely, suppose we have a factorization \(X \xrightarrow{f_1} I \xrightarrow{f_2}\Lambda\). This yields a diagram
\[
\xymatrix{
X \ar[d]\ar[r]_{f_1} & I \ar@{-->}[d]\ar[r]_{f_2} & \Lambda\ar[d] \\
\Hom(\Lambda, \RR) \ar[r]_{f_2^*} & \Hom(I, \RR) \ar[r]_{f_1^*} & \Hom(X, \RR)
}
\]
and we claim that we can fill in the dashed arrow to yield a principally polarized (tropical) abelian variety \(P = \Hom(I, \RR)/I\) which will then necessarily satisfy the conditions in \(S_2\). But this is easy: if we let \(j : X \to \Hom(\Lambda, \RR)\) be the inclusion of the lattice, and let \(d\) be the index \([I : f(X)]\), then we can simply define the map \(k : I \to \Hom(I, \RR)\) as
\[
k = \frac{1}{d}f_2^* \circ j \circ f_1^\dag.
\]
One can easily check that this makes the diagram commute, and moreover, the resulting tropical abelian variety is by definition principally polarized. This yields our inverse map \(S_3 \to S_2\).
\end{proof}

\begin{proof}[Proof of Proposition \ref{prop_12}]
This proof is roughly the same as in the proof of Theorem 3.2 of \cite{gottsche}. We begin by noting that the universal property of Jacobians (which is still valid in our case) yields a map \(S_1/_\sim \to S_2\), where we say that two maps in \(S_1\) are equivalent if they differ by translation in \(A\).

So consider an element \(F : P \to A\) of \(S_2\). Every principally polarized abelian variety of dimension \(g = 2, 3\) has a natural line bundle together with a section, which we call (by analogy with the complex case) the \(\Theta\)-function (see as usual \cite[Section 5.2]{MZ_tropical_jacobians}). This can be explicitly written as
\[
\Theta(x) = \max_{\lambda \in I} \big\{Q(\lambda, x) - \tfrac{1}{2}Q(\lambda, \lambda)\big\}
\]
where \(Q\) is the bilinear form on \(I\) resulting from the principal polarization. The Voronoi decomposition of \(\Hom(I, \RR) \cong \RR^g\) provides an \(I\)-periodic tiling of \(\RR^g\), whose 1-skeleton is a genus \(g\) tropical curve when projected to the quotient torus \(P\). The inclusion followed by the map \(F : P \to A\) yields the map \(f : \Gamma \to A\) that we desire\footnote{Note that this satisfies the conditions of Definition \ref{def_tropical_map} since the map \(\Gamma \to J(\Gamma)\) does.}. Moreover, since we are assuming that \(A\) is Torelli, it follows that the maps \(S_1/_\sim \to S_2\) and \(S_2 \to S_1/_\sim\) are inverse to each other (since \(P\) is isomorphic to the Jacobian \(J(\Gamma)\)).

Finally, all of this is up to translation in \(A\). We finally need to look at those elements  \(a \in A\) such that \(t_a^*L \cong L\). However, by Remark \ref{rmk_translation}, we see that this is the case if and only if \(a \in \ker (A \to \hat{A})\), whence the claim.
\end{proof}

\subsection{Multiplicity Computation}\label{subsec-multcon}
\subsubsection{Tropicalization of abelian varieties} We will first explain what we mean by \emph{tropicalization} of abelian varieties over $\mathbb{K}$. Our discussion follows closely \cite{BakRab}. To be precise, we allow arbitrary (non-degenerate) polarizations, but we restrict ourselves to abelian varieties uniformized by tori, as this is the only case we need. We continue to use the notation and terminology introduced in Section \ref{subsec:rigidunif}.

Let $ A $ be an abelian $\mathbb{K}$-variety of dimension $g$, which is uniformized by the data $ M \subset T$. The points $T(\mathbb{K})$ can be naturally identified with the group $ \mathrm{Hom}(\hat{M}, \mathbb{K}^*) $. Via the group homomorphism
$$ - \mathrm{log} \colon \mathrm{Hom}(\hat{M}, \mathbb{K}^*) \cong T(\mathbb{K}) \to \mathbb{R}^g $$
defined by $ (t_1, \ldots, t_g) \mapsto - (\mathrm{log} \vert t_1 \vert, \ldots, \mathrm{log} \vert t_g \vert) $,
we can identify $ M(\mathbb{K})$ with a full rank lattice $ \Lambda = - \mathrm{log}(M) $ in $ \mathbb{R}^g $. 

One easily checks that a polarization $c_L$ yields a polarization of the tropical torus $A^{\mathrm{tr}} = \mathbb{R}^g/\Lambda$, which we will continue to denote $c_L$. Thus, a pair $(A,L)$ tropicalizes to a tropical polarized abelian variety in the sense of Definition \ref{definition-tropicalAV}. In conclusion, we obtain a functor 
$$ (A,L) \mapsto (A^{\mathrm{tr}}, c_L). $$

We shall need one additional, crucial, fact from \cite{BakRab}. Let $C$ be a smooth projective and connected $\mathbb{K}$-curve of genus $g$, and let $J(C)$ denote its Jacobian variety. Let moreover $C^{\mathrm{tr}}$ denote the tropicalization of $C$, i.e., its \emph{minimal skeleton} in the sense of non-Archimedean geometry. Then one has a canonical isomorphism
$$ J(C^{\mathrm{tr}}) \cong J(C)^{\mathrm{tr}}$$
as principally polarized tropical abelian varieties (see \cite[2.9, 2.10]{BakRab}).

\begin{rmk}
The reader will find a closely related construction of tropicalization of abelian varieties in \cite{viv_torelli}, where the author works over a complete non-archimedean field with a discrete valuation, which allows the use of N\'eron models (instead of uniformization). 
\end{rmk}

\subsubsection{Analysis of multiplicities}

We can understand the multiplicity computation as follows. Let \(A\) be an abelian variety defined over \(\KK\) and let \(A^{\mathrm{tr}}\) be its tropicalization. Let \(f : \Gamma \to A^{\mathrm{tr}}\) be a tropical morphism. The multiplicity is then the number of curves \(C\) which tropicalize to \(\Gamma\) and which fit into the diagram
\[
\xymatrix{
C \ar@{-->}[rr]\ar@{-->}[d] & & A \ar[d] \\
\Gamma \ar[rr] & & A^{\mathrm{tr}}.
}
\]
That is, how many ways can we fill in the dashed corner? If we now examine the proof of Proposition \ref{prop_naive_count}, we see that it consists of two separate lifting problems. These are
\[
\xymatrix{
C \ar@{-->}[rr]\ar@{-->}[d] & & P \ar[d] & & P \ar@{-->}[rr]\ar@{-->}[d] & & A\ar[d]\\
\Gamma \ar[rr] & & J(\Gamma) & & P^{\mathrm{tr}} \ar[rr] & & A^{\mathrm{tr}}
}
\]
respectively, where \(P^{\mathrm{tr}}\) is a tropical principally polarized abelian variety (and \(P\) is a choice of a lift of it to \(\KK\)).

\begin{lem}
Let \(\Gamma\) be a tropical curve and let \(J(\Gamma)\) be its Jacobian. Let \(P\) be a principally polarized abelian variety defined over \(\KK\) which tropicalizes to \(J(\Gamma)\). Then there is a unique dashed lift that fits into the diagram
\[
\xymatrix{
C \ar@{-->}[rr] \ar@{-->}[d] && P\ar[d] \\
\Gamma \ar[rr] && J(\Gamma).
}
\]
\end{lem}

\begin{proof}
This relies on two facts. If \(P\) is a principally polarized abelian variety (over \(\KK\)) of dimension \(g = 2, 3\), then \(P = J(C)\) for a unique curve \(C\). As explained above, we have that \(J(C)^{\mathrm{tr}} = J(C^{\mathrm{tr}})\). Since we are assuming that \(A\) is Torelli, it follows that \(C^{\mathrm{tr}} = \Gamma\), from which it follows that the diagram can be completed uniquely. 
\end{proof}

For the final multiplicity computation, recall that the collection of maps \(P^{\mathrm{tr}} \to A^{\mathrm{tr}}\) is in bijection with the set \(\ker (A \to \hat{A}) \times S_3\).

\begin{lem}\label{lem_mult}
Let \((v, f_1, f_2)\) be an element of \(\ker(A^{\mathrm{tr}} \to \hat{A}^{\mathrm{tr}}) \times S_3\), where \(f_2 \circ f_1 = c_L^\dag\). Let \(G = I/f_1(X)\). Then the number of lifts of \(P^{\mathrm{tr}} \to A^{\mathrm{tr}}\) to \(\KK\) is given by
\[
n \cdot \# \Hom^{sym}(G, G^*).
\]
\end{lem}

\begin{proof}
Our goal is to determine the number of dashed arrows that fit into the diagram
\begin{equation}\label{eq_fill_in}
\xymatrix{
X \ar[r]^{f_1} \ar[d] & I \ar[r]^{f_2} \ar@{-->}[d] & \Lambda \ar[d] \\
\Hom(\Lambda, \KK^\times) \ar[r]^{f_2^*} \ar[d] & \Hom(I, \KK^\times) \ar[r]^{f_1^*} \ar[d] & \Hom(X, \KK^\times) \ar[d] \\
\Hom(\Lambda, \RR) \ar[r]^{f_2^*}  & \Hom(I, \RR) \ar[r]^{f_1^*}  & \Hom(X, \RR),
}
\end{equation}
as well as the number of lifts \(\ker(A \to \hat{A}) \to \ker(A^{\mathrm{tr}} \to \hat{A}^{\mathrm{tr}})\).

We will without loss of generality assume that \(v = 0\), so that the map \(P^{\mathrm{tr}} \to A^{\mathrm{tr}}\) is a homomorphism.

For the next part of the proof, we will regard a map \(I \to \Hom(I, \KK^\times)\) as a bilinear map \(I \times I \to \KK^\times\). In particular, we are thus trying to determine the number of (symmetric) bilinear maps \(I \times I \to \KK^\times\) (which make the diagram \eqref{eq_fill_in} commute, when viewed as a map \(I \to \Hom(I, \KK^\times)\)). Let \(B\) be the set of such maps, and let \(b_0 \in B\) be fixed.

We claim that there is a bijection \(B \to \Hom^{sym}(G, G^*)\). There is a map \(\Hom^{sym}(G, G^*) \to B\) given as follows. Let \(\varphi : G \times G \to \KK^\times\). Then we map \(\varphi\) to
\[
(i_1, i_2) \mapsto b_0(i_1, i_2)\varphi([i_1], [i_2])
\]
which is an element of \(B\).

To construct the inverse, choose \(b \in B\) arbitrary, and define a map \(\varphi_b : G \times G \to \KK^\times\) as
\[
(g_1, g_2) \mapsto \frac{b(\overline{g_1}, \overline{g_2})}{b_0(\overline{g_1}, \overline{g_2})}
\]
where \(\overline{g_i}\) is any lift of \(g_i\) to \(I\). As any pair of lifts only varies by an element of \(X\); since the functions in \(B\) agree on elements of \(X\), it follows that this is well defined. Finally, since the functions in \(B\) are symmetric, it follows that so are the resulting functions on \(G\).

Finally, note that \(K = \ker(A \to \hat{A})\) and \(K^{\mathrm{tr}} = \ker(A^{\mathrm{tr}} \to \hat{A^{\mathrm{tr}}})\) fit into the diagram (via the Snake Lemma)
\[
\xymatrix{
0 \ar[r]& \ker(c_L^*) \ar[r]\ar[d] & K \ar[r]\ar[d] & K^{\mathrm{tr}} \ar[d] \ar[r] & 0\\
0 \ar[r] & \Hom(X, \mu_n) \ar[r]\ar[d]_{c_L^*} & A[n] \ar[r]\ar[d] & A^{\mathrm{tr}}[n] \ar[r]\ar[d] & 0 \\
0 \ar[r] & \Hom(\Lambda, \mu_n) \ar[r] & \hat{A}[n] \ar[r] & \hat{A^{\mathrm{tr}}}[n] \ar[r] & 0. 
}
\]
Since the bottom two rows are split exact, it follows that \(K \to K^{\mathrm{tr}}\) is surjective, and since \(\ker(c_L^*)\) has \(n\) elements, it follows that translations contribute a factor of \(n\). In particular, the total multiplicity is \(n \cdot \#\Hom^{sym}(G, G^*)\) as claimed.

\end{proof}

We can now combine these facts to prove Theorem \ref{thm_main_A}.

\begin{proof}
We have that the collection of morphisms from tropical curves to \(A\) in the linear system \(|L|\) is in bijection with the set \(S_1 \iff \ker(A \to \hat{A}) \times S_3\), each of which contributes multiplicity according to Lemma \ref{lem_mult}. Thus the total count of such curves is
\begin{align*}
\sum_{(v, f, g) \in S_1} m_{(v,f,g)} &= \sum_{v \in \ker(A \to \hat{A})}\sum_{H \leq \Lambda/c_L^\dag(X)} n\cdot \#\Hom^{sym}(H, H^*) \\
 &= n \cdot |\ker(A \to \hat{A})| \sum_{H \leq \Lambda/c_L^\dag(X)} \#\Hom^{sym}(H, H^*) \\
 &= n^2 \nu^{\dag}(d_1, \ldots, d_g)
\end{align*}
as claimed.

\end{proof}

There is an important corollary hiding in the above work.

\begin{cor}
Let \((A, c)\) and \((A', c')\) be simple Torelli polarized tropical abelian varieties with polarization of type \((d_1, \ldots, d_g)\). Then the count of genus \(g\) curves in \((A,c), (A',c')\) coincide.
\end{cor}

\begin{proof}
We note that the count in the above theorem depends only on the type \((d_1, \ldots, d_g)\) and makes no reference to \(A\) or \(c\) beyond that, hence the claim.
\end{proof}

\section{Some specific computations}

We will now compare these to other known results and provide a few numerical computations.

\begin{thm*}[Theorem \ref{thm_main}]
Let \((A,c_L)\) be a simple polarized abelian surface with polarization of type \((1, n)\). Then the number of genus 2 curves in \(|L|\) is \(n^2\sigma_1(n)\).
\end{thm*}

\begin{proof}
It suffices to show that \(\nu(1, n) = \sigma_1(n)\). But this is clear: the number of symmetric morphisms from \(\ZZ/d\ZZ\) to \((\ZZ/d\ZZ)^*\) is \(d\), since such homomorphisms are trivially symmetric. Thus
\[
\sum_{H \leq \ZZ/n\ZZ} \# \Hom^{sym}(H, H^*) = \sum_{d \mid n} d = \sigma_1(n)
\]
as claimed.
\end{proof}

We also have the following case.

\begin{prop}
Let \(p\) be a prime. Then
\[
\nu(p, pn) = \sigma_1(p^2n) + p^3 \sigma_1(n)
\]
\end{prop}

\begin{proof}
We separate the proof into two parts. Let \(G = \ZZ/p\ZZ \times \ZZ/(pn)\ZZ\). First, we note that the cyclic subgroups satisfy
\[
\sum_{\substack{H \leq G \\ H \text{ is cyclic}}} \#\Hom(H, H^*) = \sigma_1(p^2n)
\]
(Since both sides are multiplicative, this can be reduced to the case \(n = p^\ell\), which is easy). Next, we have that the non cyclic subgroups (which are all of the form \(\ZZ/p\ZZ \times \ZZ/(pd)\ZZ\) for some \(d \mid n\)) satisfy
\begin{multline*}
\Hom^{sym}\big(\ZZ/p\ZZ \times \ZZ/(pd)\ZZ, (\ZZ/p\ZZ \times \ZZ/(pd)\ZZ)^*\big) \\
= \Hom\big(\ZZ/p\ZZ, (\ZZ/p\ZZ)^*\big) \times \Hom\big(\ZZ/p\ZZ, (\ZZ/(pd)\ZZ)^*\big) \times \Hom\big(\ZZ/(pd)\ZZ, (\ZZ/(pd)\ZZ)^*\big)
\end{multline*}
which has order \(p^3d\). Summing over all of these yields \(p^3\sigma_1(n)\) as claimed.
\end{proof}

It follows from this that the generating function \(\sum_{n=1}^\infty \nu(p, pn) q^n\) can be written as
\[
\sum_{n=1}^\infty \nu(p, pn) q^n = p^3E_2(\varepsilon^{p^2}) + E_2(\varepsilon) + \frac{p^3 + 1}{24}
\]
where \(q = \varepsilon^{p^2}\), and where \(E_2(q) = -\frac{1}{24} + \sum_{n=1}^\infty \sigma_1(n)q^n\) is the Eisenstein series of weight 2. In particular, these generating functions are quasimodular forms (if not of pure weight) for a non-trivial congruence subgroup of \(SL_2\ZZ\). Moreover, one suspects that there is a consistent definition of \(\nu(p, 0)\) which would allow us to account for the constant term; this could possibly be further explored using the language of tropical Gromov-witten theory.

For other cases, we produce the following table of values (which have been computed in a custom SAGE program). Note that this does {\em not} agree with the table in \cite{LS_curves_genus_g}, which has some numerical errors (which were also noted in \cite{BOPY_curves}). Furthermore, unlike the values \(\nu(1,d) = \sigma_1(d)\), these do not form a recognizable sequence as, say, the coefficients of a modular form.

\begin{center}
\begin{tabular}{c|c|c||c|c}
\(d\) & \(\nu(d,d)\) & \(\nu(d,d,d)\) &\((d_1, d_2)\) & \(\nu(d_1, d_2)\) \\
\hline
2 & 15 & 135 & (2, 4) & 39\\
3 & 40 & 1120 & (2, 6) & 60\\
4 & 151 & 11,287 & (2, 8)& 87\\
5 & 156 & 19,656 & (2, 10)& 90\\
6 & 600 & 151,200 & (2, 12)& 156\\
7 & 400 & 137,600 & (3, 6)& 120\\
8 & 1335 & 810,135 & (3, 9)& 148\\
9 & 1201 & 915,853 & (3, 12)& 280\\
10 & 2340 & 2,653,560 & (4, 8)& 375\\
11 & 1464 & 1,950,048 & (4, 12)& 604\\
12 & 6040 & 12,641,440 & (4, 16)& 823\\
13 & 2380 & 5,231,240 & (5, 10)& 468\\
14 & 6000 & 18,576,000 & (5, 15)& 624\\
15 & 6240 & 22,014,720 & (6, 12)& 1560\\
16 & 11,191 & 54,681,751 & (6, 18)& 2220\\
\end{tabular}
\end{center}

\section{Conclusion and further work}

The goal of this paper was to work towards understanding how the count of curves in abelian varieties adapts to the tropical setting. In the complex setting, this has been studied e.g. in \cite{BL_ab}, \cite{LS_curves_genus_g}, and \cite{Rose_thesis_paper}.

In future work we intend to look at extending this to higher genera in a number of potential ways. One could naturally look at developing a tropical analogue of \cite{BL_ab}, which would presumably require a version of reduced tropical Gromov-Witten theory.

One could instead focus on hyperelliptic curves, using the work of \cite{chan_hyperelliptic} on tropical hyperelliptic curves. One should expect some interesting results either by trying to produce an analogous argument to the one provided in this paper, combined with an understanding of the Jacobians of tropical hyperelliptic curves.

One could furthermore follow in the vein of \cite{Rose_thesis_paper} and look at trying to understand tropical genus 0 curves in the associated tropical Kummer surface to \(A\). This would require an understanding of the relationship between the count of such curves and the count of hyperelliptic curves in \(A\); in \cite{Rose_thesis_paper} this was provided by using orbifold Gromov-Witten theory, which to the best of our knowledge has not been developed in the tropical setting.

\bibliographystyle{alpha}
\bibliography{cheb}

\end{document}